\documentclass[11pt]{article}
\usepackage{amsmath,amssymb,amsthm}
\usepackage{ed}
\usepackage{enumitem} 
\usepackage{geometry}
\usepackage{graphicx}
\usepackage{multirow}
\usepackage{array}
\usepackage{wrapfig}
\newcommand{\R}{\mathbb{R}}
\newcommand{\calV}{\mathcal{V}}

\newcommand{\bfi}{\bfseries\itshape}

\newtheorem{proposition}{Proposition}
\newtheorem{remark}{Remark}
\newtheorem{theorem}{Theorem}
\newtheorem{lemma}{Lemma}
\newtheorem{corollary}{Corollary}

\begin{document}

\title{Data-Specific Hyper-Parameter Design: A Paradigm Shift in Reservoir Computing}
\author{G Manjunath$^{1}$, Juan-Pablo Ortega$^{2}$, and Alma van der Merwe$^{1}$}

\makeatletter
\addtocounter{footnote}{1} \footnotetext{%
University of Pretoria. Department of Mathematics and Applied Mathematics. Pretoria 0002, South Africa.   {\texttt{manjunath.gandhi@up.ac.za \& }\texttt{alma.vandermerwe@up.ac.za} }}
\addtocounter{footnote}{1} \footnotetext{%
Nanyang Technological University. Division of Mathematical Sciences. School of Physical and Mathematical Sciences. Singapore.   {\texttt{juan-pablo.ortega@ntu.edu.sg} } }
\makeatother

\date{\today}
\maketitle
\begin{abstract}
Reservoir computing typically relies on large, randomly generated reservoirs, enabling simple, often linear readouts. Over the past two decades, most constructions have exploited the freedom to select the reservoir, constrained primarily by stability conditions based on state contraction or memory capacity. However, these designs are largely independent of the input data and learning objective, resulting in a trial-and-error methodology driven by randomness. In high dimensions, the reservoir acts as a random embedding of the input history, implicitly relying on Johnson--Lindenstrauss--type concentration phenomena to preserve information and maintain numerical stability.

In contrast, we develop reservoir design principles from a geometric perspective for inputs generated by deterministic dynamical systems. Rather than relying on random embeddings, we require reservoir state increments to align within a cone around an input-determined vector subspace, and prove that such a cone concentration reduces ridge-regression training error. When the cone angle is small, the variance of reservoir states concentrates in the input-determined subspace, improving conditioning of the empirical second-moment matrix and strengthening alignment between dominant covariance directions and the state-target cross-covariance.

For echo state networks, we provide a constructive approach to reservoir design. The reservoir matrix is chosen so that associated Krylov-chain directions remain nearly closed within an input-determined subspace while permitting controlled mixing in its orthogonal complement. We also provide a spectral diagnostic for ridge regression training that identifies when reservoir geometry concentrates predictive information into a few dominant covariance modes and when ``spectral pollution'' inhibits forecasting. Numerical experiments demonstrate consistent performance gains over arbitrary reservoir constructions.
\end{abstract}

\section{Introduction} \label{sec:introduction}
We consider the general discrete-time state-space model
\begin{equation}
\label{eq:state_space}
x_{n+1} = g(u_n, x_n), \qquad y_n = r(x_n),
\end{equation}
where $u_n$ denotes the {\bfi input}, $x_n$ the internal {\bfi state} belonging to some metric or topological spaces $U$ and $X$ respectively, and $g$ is a continuous map. The elements $ y_n$ are called the {\bfi outputs} or {\bfi targets}.
Reservoir computing (RC) \cite{maass1, Jaeger04, jaeger2001} constructs the state equation $g$ randomly and then trains a readout map $r$ to approximate a target sequence $y_n$ using only the state $x_n$.  A central virtue of RC lies in its potential for
computational efficiency: training is typically reduced to solving a
linear regression problem, often ridge-regularized, while the internal
dynamics are left untrained. Additionally, when used for the forecasting of deterministic dynamical systems, RC has advantages over other data-driven methods \cite{manjunath2025universal} and generalizes the Takens delay embedding \cite{hart:ESNs, RC18, RC21}.

A particular case of RC is an echo state networks (ESN) \cite{Matthews:thesis, Matthews1993}, where the update map takes the form 
\begin{equation}
\label{eq:esn_update}
x_{n+1} = \tanh(A u_n + B x_n), \qquad y_n = r(x_n),
\end{equation}
with matrices $A$ and $B$ (usually called {\bfi  input} and {\bfi connectivity matrices}, respectively) of appropriate dimension and where the nonlinear function $\tanh $ is applied componentwise; the readout $r$ is customarily realized as a linear transformation $W x_n$ where $W$ is a matrix. From a theoretical viewpoint, universal approximation results \cite{RC7, RC8, RC20, RC12} highlight the relevance of their study for practical applications. Practical architecture and hyperparameter selection methods for ESNs are not yet established, and in particular, designing the reservoir matrix $B$ remains a major practical challenge.
The difficulty stems from two sources: the high dimensionality of the
reservoir space and the time-dependent forcing induced by the input.
The resulting dynamics are typically intractable, and reservoir matrices
are therefore chosen randomly, subject mainly to spectral radius
constraints ensuring stability, or in other words, guaranteeing the {\bfi echo state property}, that there is exactly one bi-infinite sequence $\{x_n\}$ for each $\{u_n\}$ that satisfies the state space equation. 

When the input sequence $\{u_n\}\subset U$ is deterministic in the sense
that there exists $J>0$ and a well-defined map
\[
F: (u_{n-J-1},\ldots,u_n) \mapsto u_{n+1}
\]
valid for all $n$, the situation simplifies conceptually.
One may then regard the input history as evolving on a subspace of $U^J$,
so that the problem reduces to embedding this $J$-dimensional structure
into the reservoir state space \cite{hart:ESNs, allen:tikhonov, RC18, RC21, RC26}.
Recent developments show that under suitable conditions, such embeddings
can even be made bi-Lipschitz, implying that exact reconstruction of the
underlying dynamics is possible in principle since a topological embedding would ensure the readout would satisfy $r(x_n) = y_{n}$, when $y_n = u_n$ \cite{RC36}. 

However, the strength of reservoir computing does not lie in exact
reconstruction via a complex nonlinear inversion and learning, but rather in the
availability of cheap training procedures to determine $r$.
In practice, the readout is typically determined by ridge regression,
and performance depends critically on the geometry of the empirical
second-moment matrix of the states.
Thus, the central question is not merely whether the input can be
embedded, but whether the reservoir state geometry is favorable for
stable and efficient linear training.

In this work, we provide a principled design of the reservoir matrix for
deterministic inputs that aims to shape the reservoir state geometry
toward optimal ridge regression performance. In the following subsection, we provide examples to motivate the need for this shaping.

\subsection{Geometry of the ridge regression}

All along this paper, we assume that the state-space model \eqref{eq:state_space} satisfies the \textit{echo state property} (ESP) \cite{jaeger2001echo, Manjunath:Jaeger, RC31, RC32} so that for any input sequence $\{u_n\}$ there exists a unique state sequence $\{x_n\}$ satisfying \eqref{eq:state_space} for all $n \in \mathbb{Z}$. This uniqueness is often guaranteed by imposing contractivity of $g$ with respect to the $x$-variable. 

\paragraph{Washout and independence of initial conditions.} 
Let $x_n(x_0)$ denote the reservoir state initialized at $x_0$ and driven by $\{u_n\}$. Under the ESP, there exists a unique bi-infinite sequence $\{x_n\}_{n \in \mathbb{Z}}$ satisfying \eqref{eq:state_space}, and by the {\bfi uniform attraction property (UAP)} \cite{manjunath2022embedding}, we have $x_n(x_0) \to x_n$ uniformly in $x_0$. Hence, after a sufficiently long washout, the state is effectively independent of initialization. All empirical moments are computed after such a washout, so that the quantities in the theorems are approximated with negligible error, whereas the theoretical results assume the exact trajectory $\{x_n\}$.

Let $\{x_n\}_{n=1}^T\subset\mathbb{R}^N$ be state vectors corresponding to a given input sequence and let $\{y_n\}_{n=1}^T\subset\mathbb{R}^p$ be targets. The input and target sequences are assumed to be realizations of a {\bfi stationary process}, which, in the presence of the ESP, implies the stationarity of the state sequences \cite{Manjunath:Jaeger, RC15}. This is a typical situation encountered when the inputs are observations of a dynamical system preserving an invariant measure, and the target is a measurable function of the input. 
For simplicity in notation, here and throughout without further mention, we assume the sequences are centered and define the empirical state covariance and state--target cross-covariance by
\[
\Sigma := \frac{1}{T}\sum_{n=1}^T x_n x_n^\top \in \mathbb{R}^{N\times N},
\qquad
C := \frac{1}{T}\sum_{n=1}^T x_n y_n^\top \in \mathbb{R}^{N\times p}.
\]
The ridge regression with regularization parameter $\lambda>0 $ solves the optimization problem
\[
\min_{W\in\mathbb{R}^{N\times p}}
\frac{1}{T}\sum_{n=1}^T \|W^\top x_n - y_n\|_2^2
+ \lambda \|W\|_F^2.
\]
Its solution is unique and satisfies
\[
(\Sigma + \lambda I) W = C.
\]
Let $\Sigma = \sum_{i=1}^N \lambda_i q_i q_i^\top$ with $\lambda_1\ge\cdots\ge\lambda_N\ge0$.
The ridge solution admits the expansion
\[
W = \sum_{i=1}^N \frac{q_i q_i^\top C}{\lambda_i+\lambda}.
\]
Thus, the informativeness of an eigendirection $q_i$ is measured by
$\|q_i^\top C\|_2$, and the ridge regression is stable precisely when directions
with large variance $\lambda_i$ also satisfy that $\|q_i^\top C\|_2$ is large.
If large eigenvalues occur in directions with negligible $\|q_i^\top C\|_2$, we say that the ridge regression suffers from {\bfi spectral pollution}. In the particular case, when $y_n$ is a scalar, then 
\[
w = \sum_{i=1}^N \frac{\langle q_i,c\rangle}{\lambda_i+\lambda}\,q_i,
\]
where $c = \frac1T\sum_{n=1}^T x_n y_n$.  Clearly, for the numerator not to become insignificant while the denominator is not, then $c$ should lie in a space 
spanned by a collection of $q_i$'s, where $q_i$'s are those eigenvectors corresponding to $\lambda_i$'s that are large.  We now present two toy examples that illustrate two different ways in which ridge regression could go wrong and, more explicitly, show that ridge regression succeeds or fails depending on how variance is organized in the state space.

\paragraph{Example 1 (Causal RC state: large-variance nuisance direction and signal attenuation).} Let $\{u_n\}$ be a centered stationary deterministic input sequence and let the prediction target be $y_n=u_n$. Suppose that the two-dimensional reservoir state is given by
\[
x_n=(\sigma\xi_n,\;u_{n-1})^\top,
\]
where $\xi_n$ is the state of a \emph{internal reservoir mode},
normalized so that $\langle \xi_n^2\rangle=1$, with $\langle\cdot\rangle$ denoting empirical averaging. Assume 
$|\varepsilon|\ll1$ and $\sigma^2\gg1$. 
Let
\[
\varepsilon:=\langle \xi_n u_n\rangle,
\qquad
\rho:=\langle u_{n-1}u_n\rangle,
\]
Then
\[
\Sigma=\langle x_n x_n^\top\rangle
=
\begin{pmatrix}
\sigma^2 & \sigma\varepsilon\\
\sigma\varepsilon & 1
\end{pmatrix}
\quad \mbox{and} \quad
c=\langle x_n y_n\rangle
=(\sigma\varepsilon,\;\rho)^\top.
\]
Since $\sigma^2\gg1$, $\Sigma$ is poorly conditioned and this calls for
ridge regression.  The ridge regression with parameter $\lambda>0$ yields (since $|\varepsilon|\ll1$, we retain only terms up to first order in $\varepsilon$  neglect $O(\varepsilon^2)$ contributions) 
\[
\hat w=(\Sigma+\lambda I)^{-1}c
=\Bigl(\frac{\sigma\varepsilon}{\sigma^2+\lambda},\;
\frac{\rho}{1+\lambda}\Bigr)^\top.
\] 
Assuming that $\sigma^2\gg\lambda$, the Taylor series expansion of the first entry of $\hat w $ in powers of $\lambda/ \sigma^2 $ yields
\[
\frac{\sigma\varepsilon}{\sigma^2+\lambda}
=
\frac{\varepsilon}{\sigma}
-
\frac{\varepsilon\lambda}{\sigma^3}
+
O\!\left(\frac{\varepsilon\lambda^2}{\sigma^5}\right).
\]
Hence,
\[
\hat w
=
\Bigl(
\frac{\varepsilon}{\sigma},\;
\frac{\rho}{1+\lambda}
\Bigr)^\top
+
O\!\left(\frac{\varepsilon\lambda}{\sigma^3}\right)e_1.
\]
Consequently,
\[
\hat y_n
=
\hat w^\top x_n
=
\varepsilon\xi_n
-
\frac{\varepsilon\lambda}{\sigma^2}\xi_n
+
\frac{\rho}{1+\lambda}u_{n-1}
+
O\!\left(\frac{\varepsilon\lambda^2}{\sigma^4}\right).
\]
While the nuisance contribution is $O(\varepsilon)$ which is desired,
the informative component that contains the autocorrelation $\rho$ between the input and the target is deterministically attenuated
by the factor $(1+\lambda)^{-1}$.
Geometrically, a high-variance reservoir direction with weak task alignment forces ridge regression to shrink the causal informative direction based on variance magnitude rather than predictive relevance. This disadvantage has serious negative implications when, for instance, the resulting model is used in the iterated multi-step prediction of dynamical systems. Indeed, it has been shown \cite{berry2022learning, RC22} that in that situation, estimation errors are amplified exponentially at a rate ruled by the top Lyapunov exponent of the estimated system. 

\paragraph{Example 2 (Low-variance informative direction and over-regularization of signal).}
Retain the notation of Example~1 with the same target, and suppose that
\[
x_n=(\sigma\xi_n,\;\alpha u_{n-1})^\top,
\]
with $0<\alpha\ll1$ and $\sigma^2\gg1$.
Assume $\langle \xi_n u_n\rangle=0$ so that
$\varepsilon=0$ and $\rho=\langle u_{n-1}u_n\rangle\neq0$. Then
\[
\Sigma
=
\begin{pmatrix}
\sigma^2 & 0\\
0 & \alpha^2
\end{pmatrix},
\qquad
c
=
(0,\;\alpha\rho)^\top.
\]
The ridge regression (necessary due to ill-conditioning of $\Sigma$ due to $\sigma^2\gg1$) yields
\[
\hat w
=
(\Sigma+\lambda I)^{-1}c
=
\Bigl(
0,\;
\frac{\alpha\rho}{\alpha^2+\lambda}
\Bigr)^\top.
\]
Hence
\[
\hat y_n
=
\hat w^\top x_n
=
\frac{\alpha^2}{\alpha^2+\lambda}\,\rho\,u_{n-1}.
\]
If $\lambda\gg\alpha^2$ then
\[
\frac{\alpha^2}{\alpha^2+\lambda}
=
\frac{\alpha^2}{\lambda}
+
O\!\left(\frac{\alpha^4}{\lambda^2}\right)
\ll1,
\]
and the predictive signal is strongly attenuated here as well, despite the situation being the opposite of the previous example, in that the informative direction has small variance.  This suggests that ridge regression can almost eliminate the target variable if the regularization parameter is used to control the large-variance nuisance direction.  

We note that in classical (unstructured) least-squares problems, it is standard to
require the empirical covariance $\Sigma$
to be well-conditioned, that is, to have a moderate condition number $\lambda_{\max}/\lambda_{\min}$.  This requirement arises because small eigenvalues amplify noise in the
solution of the normal equations and lead to numerical instability; see,
for example, \cite{belsley1980regression}.  In contrast, our approach in this paper is to mitigate spectral pollution (the representation problem) while also capturing the dynamics of the input within the structure of the reservoir (the dynamical problem).


\subsection{Our approach}

Attempting to tune individual eigenvectors of $\Sigma$ to align with the target, so that problems of the type illustrated in the previous examples are avoided, can be notoriously difficult. However, we can instead shape the entire reservoir geometry so that the covariance spectrum becomes effectively low-dimensional and the cross-covariance between the target and the states is forced to lie predominantly in the same subspace since $c$ lies in the space spanned by the state vectors $x_n$.  These ideas are spelled out in Theorems~\ref{thm:approx_covleak_scalar} and~\ref{thm:novanish_scalar}, which show that directions selected by large variance cannot have negligible alignment with the target when a certain cone aperture is small which, as we will see in Sections \ref{Construction of the linear reservoir} and \ref{sec: general reservoirs} can be used, in the case of the ESN systems introduced in \eqref{eq:esn_update} to device a data-driven design principle of the connectivity $B$. We define the cone geometry next. We first work with linear systems to explain the shaping of the geometry; that is, we use state equations like \eqref{eq:esn_update}, where $\tanh $ is replaced by the identity map, with a possibly nonlinear feature mapping.


\begin{wrapfigure}{r}{0.4\textwidth}
\centering
\includegraphics[width=\linewidth]{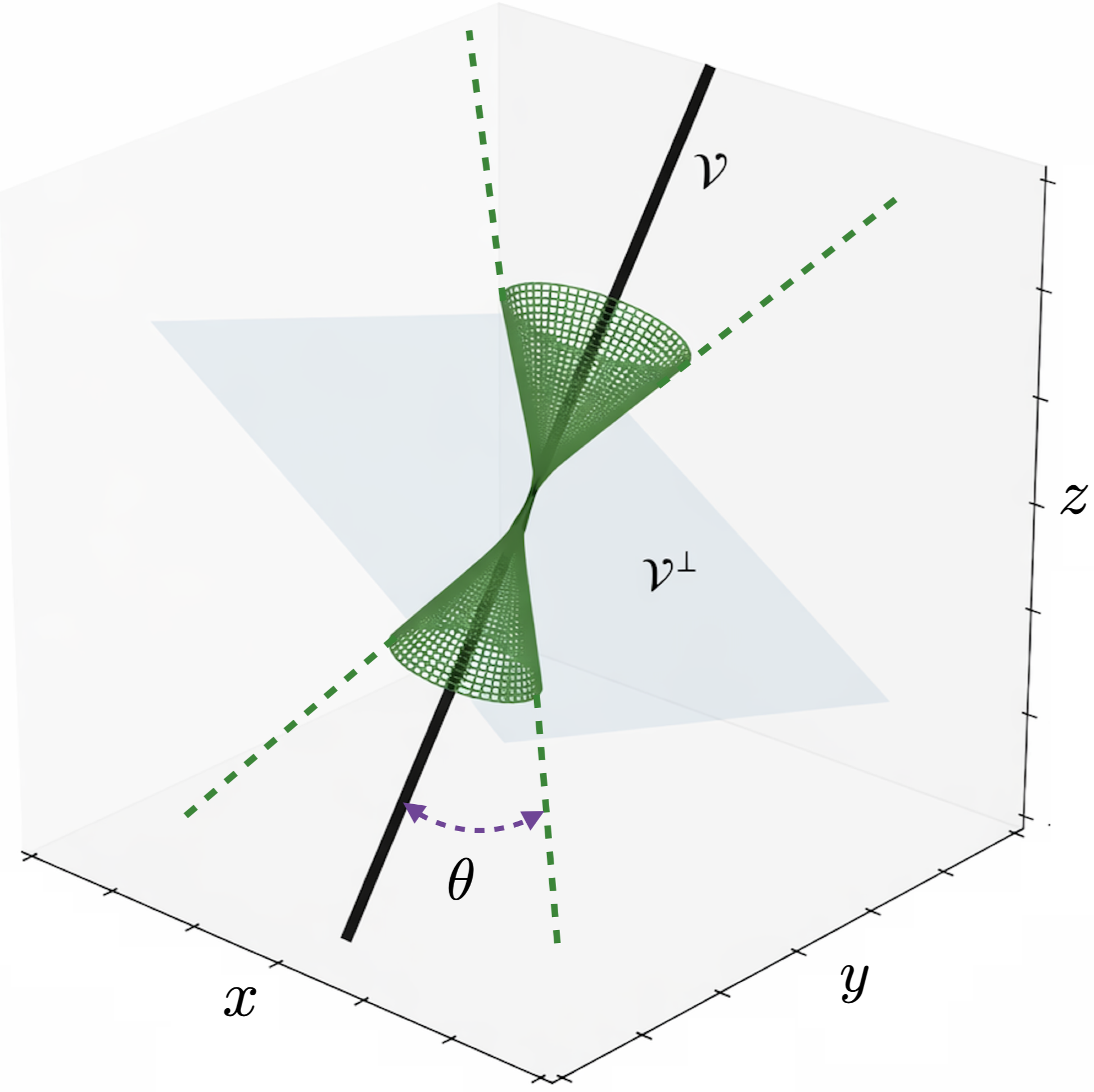}
\caption{Illustrative picture of a cone around a one-dimensional subspace $\mathcal V$ in $\mathbb{R}^3$.}
\label{fig:cone}
\end{wrapfigure}

\paragraph{State-increment subspace cone concentration and covariance leakage.}
Let $\Delta x_n := x_{n+1} - x_{n} \in \mathbb{R}^N$ denote the state increments corresponding to the solution of the system \eqref{eq:state_space} and a fixed input sequence with elements $u _n \in \mathbb{R} $. Assume also that the targets $y _n \in \mathbb{R} $ are scalars. Let $\mathcal{V} \subset \mathbb{R} ^N  $ be a vector subspace of the state space. The result below formalizes a simple geometric fact: if most state increments lie inside the cone $C(\mathcal V,\theta) $ around the vector subspace $\mathcal{V} $ with aperture angle $\theta$ (see Figure~\ref{fig:cone})  defined by 
  \[
  \mathcal C(\mathcal V,\theta)
  :=\Big\{z \in \mathbb{R}^N:\|P_{\mathcal V^\perp}z\|\le\tan\theta\,\|P_{\mathcal V}z\|\Big\},
  \]
($P_{\mathcal V}$ and $P_{\mathcal V^\perp}$ are the {\bfi  orthogonal projectors} onto the vector subspace $\mathcal V$ and its orthogonal complement $\mathcal V^\perp$), then the state-target cross-covariance $c = \frac1T\sum_{n=1}^T x_n y_n$ must also lie close to the same subspace $\mathcal{V} $. We say the increments are {\bfi cone-concentrated} around $\mathcal V$ with aperture angle $\theta$ if
$
\|P_{\mathcal V^\perp}\Delta x_n\|
\le
\tan\theta \, \|P_{\mathcal V}\Delta x_n\|
$ for ``most" $n$. We will refer to the vector $P_{\mathcal V^\perp}c$ or to its norm as the {\bfi covariance leakage}. 
When the angle $\theta$ is small, the state-increment evolution is effectively confined to a low-dimensional vector space, while considering the reservoir dimension to be very large. The next two theorems are later on proved in a vector-valued target framework in Theorems~\ref{thm:approx_covleak} and~\ref{thm:novanish_vector_robust}.

\begin{theorem}[Covariance leakage, scalar case]
\label{thm:approx_covleak_scalar}
Let $\{x_n\}_{n=0}^T\subset\mathbb{R}^N$ and 
$\{y_n\}_{n=1}^T\subset\mathbb{R}$ be centered sequences of states and targets, respectively, and define
\[
c=\frac1T\sum_{n=1}^T x_n y_n,
\qquad
\Delta x_n=x_n-x_{n-1}.
\]
Fix a subspace $\mathcal V\subset\mathbb{R}^N$. Assume:
\begin{description}
\item [(i)]  \emph{Cone condition on increments}: There exists $\mathcal G\subset\{1,\dots,T\}$, $|\mathcal G|\ge (1-\alpha)T$ and $m_\Delta>0 $ such that 
\[
\|P_{\mathcal V^\perp}\Delta x_n\|\le \tan\theta\,\|P_{\mathcal V}\Delta x_n\|,  n\in\mathcal G,
\]
and $ \|\Delta x_n\|\le m_\Delta$, for all $n$.
\item [(ii)] \emph{Uniform transverse boundedness}: Under the standing stationarity assumption (Section~1.1), there exists $M_\perp>0$ \emph{independent of $T$} such that 
\[
\|P_{\mathcal V^\perp}x_n\|\le M_\perp,\qquad n=1,\dots,T.
\]
\end{description}
Then
\[
\|P_{\mathcal V^\perp} c\| \le M_\perp\,\sigma_y, 
\qquad
\sigma_y:=\Big(\frac1T\sum_{n=1}^T y_n^2\Big)^{1/2}.
\]
In particular, when $\sigma_y$ is bounded uniformly in $T$ and $M_\perp$ is small, the leakage $\|P_{\mathcal V^\perp}c\|$ is small independently of $T$.
\end{theorem}

The next result explains why this matters for ridge regression.
Let $\Sigma=\sum_i \lambda_i q_i q_i^\top$ denote the covariance eigendecomposition.

\medskip
\noindent
\begin{theorem}[Large variance forces alignment, scalar case]
\label{thm:novanish_scalar}
Let $\Sigma=\frac1T\sum_{n=1}^T x_n x_n^\top$ be the state covariance matrix with eigenvectors $\left\{q _i\right\}_{i=1}^N\subset \mathbb{R}^N$ and let
$c=\frac1T\sum_{n=1}^T x_n y_n\in\mathbb{R}^N$ be the state-target cross-covariance.
Fix a subspace $\mathcal V\subset\mathbb{R}^N$ and denote
$c_{\mathcal V}:=P_{\mathcal V}c$ and
$c_{\mathcal V^\perp}:=P_{\mathcal V^\perp}c$.

Let $q_i$ be an eigenvector of $\Sigma$ corresponding to the eigenvalue $\lambda_i$. Assume that it satisfies
$\|P_{\mathcal V^\perp}q_i\|_2\le\varepsilon$ for some $\varepsilon\in[0,1)$,
that $\|c_{\mathcal V}\|_2$ is nonzero (nondegenerate supervision in $\mathcal V$),
and that the covariance leakage $\delta:=\|c_{\mathcal V^\perp}\|_2$ is small.
Then for each such $q_i$ it holds that
\[
|\langle q_i,c\rangle|
\ \ge\
\sqrt{1-\varepsilon^2}\,\|c_{\mathcal V}\|_2
\;-\;
\varepsilon\,\delta.
\]
In particular, if $\varepsilon$ and $\delta$ are small and
$\|c_{\mathcal V}\|_2$ is not small, then directions selected by large
variance cannot have negligible alignment with the target.  
\end{theorem}

Together, these results show that controlling the increment geometry simultaneously shapes the covariance spectrum and the ridge-relevant alignment. Towards that end, we define $v_n := A  u_n$ and $\Delta v_n := v_{n+1} - v_{n}$ and  the subspace 
\begin{equation}
\label{increment-spine}
\mathcal V := \operatorname{span}\{\Delta v_n \}
\subset \mathbb{R}^N,
\end{equation}
and call it the  {\bfi increment-spine} in state space. In practice, we shall sometimes take the span of its dominant $d$-dimensional PCA components exclusively since it forms the low-dimensional structural core along which state increments predominantly evolve. Under ergodicity, the empirical covariance of the forcing increments $v_n$ converges almost surely to its population covariance $M _{\mathcal{V}}:=\mathbb{E}[v_n v_n^\top]$. If the leading eigenspace of $M_{\mathcal{V}}$ is separated by a spectral gap, then the associated PCA subspace $\mathcal V$ is uniquely defined and is therefore identical for all typical realizations.

In order to provide further intuition about the reach of these results, we now focus on a linear system with scalar inputs $u _n $ given by
\begin{equation}
\label{linear system without phi}
x_{n+1} = A u_n + \beta B x_n, \quad \mbox{$0<\beta<1$}.
\end{equation}
The connectivity matrix $B$ does not play any role in the definition of the increment-spine space \eqref{increment-spine}, but, as we see later on in Sections \ref{Construction of the linear reservoir} and \ref{sec: general reservoirs}, this matrix will be the main tool to achieve cone-concentration in practice. In other words, we shall present a constructive methodology for designing the matrix $B$ so that the increment-spine subspace $\mathcal V $ is nearly invariant under the action of $B$. We emphasize that we do not demand exact invariance since such an invariance would not be stable under perturbations, and because $\mathcal V $ in practice is determined by a PCA cloud.  

In order to further illustrate the importance of the (approximate) $B$-invariance of $\mathcal{V} $, we start by noting that for the linear system, the increment recursion can be written as
\[
\Delta x_{n+1} = v_n + \beta B\,\Delta x_n.
\]
This shows that the transverse growth of the state increments is entirely controlled by the action of
$P_{\mathcal V^\perp}B$ since $v_n \in \mathcal V$ and hence $P_{\mathcal V^\perp} \Delta x_{n+1} =  P_{\mathcal V^\perp} \beta B\,\Delta x_n$.
By designing $B$ so that its transverse component on $\mathcal V$ is small, we enforce cone concentration of the increments and prevent persistent injection of variance outside $\mathcal V$, which ultimately leads to covariance leakage. Importantly, we shall show in Section \ref{sec: general reservoirs} that this geometric mechanism remains valid for nonlinear $\tanh$-based reservoirs like in \eqref{eq:esn_update}, that is, we shall prove that component-wise Lipschitz nonlinearities cannot create large transverse components, and therefore preserve the cone structure induced by the matrix $B$.

These considerations lead to the following result, which quantitatively formalizes the geometric insights stated above in the linear reservoir case, linking the exact invariance of $\mathcal V$, that is, $B \mathcal V \subseteq \mathcal V$, to an optimal reduction in the ridge regression training error. Later on, in Section~\ref{subsec:training_error}, we present a result (Theorem~\ref{thm:subspace_ridge_vector}) in which $\mathcal{V} $ is only approximately  $B$-invariant, and we quantify the improvement in the training error. 
\begin{theorem}[Subspace invariance and ridge training error]
\label{thm:subspace_ridge_scalar}
Consider the linear reservoir
\[
x_{n+1}=A u_n+\beta B x_n,\qquad 0<\beta<1,
\]
driven by centered scalar inputs $\{u_n\}_{n=1}^T$, $B$ is orthogonal, and the target $y_n=u_n$.
Define the empirical
moments
\[
\Sigma=\frac1{T }\sum_{n=  1}^T x_nx_n^\top,
\qquad
c=\frac1{T }\sum_{n=  1}^T x_nu_n,
\qquad
T =T-n_0 .
\]
Let $\mathcal V\subset\mathbb{R}^N$ denote the increment-spine, i.e., the $\operatorname{span}\{\Delta v_n \}$ with $v_n = A u_n$ and suppose that
\begin{equation}
\label{invariance condition}
B\mathcal V\subseteq\mathcal V.
\end{equation}
Then $c\in\mathcal V$, and for every $\lambda>0$ the ridge training error satisfies
\begin{equation}
\label{error with gain}
\min_w
\left(
\frac1{T }\sum_{n=  1}^T (w^\top x_n-u_n)^2
+\lambda\|w\|^2
\right)
=
\frac1{T }\sum_{n=  1}^T u_n^2
-
\mathcal G_\lambda,
\quad \mbox{where} \quad
\mathcal G_\lambda:=c^\top(\Sigma+\lambda I)^{-1}c.
\end{equation}
We refer to ${\cal G}_\lambda$ as the {\bfi ridge prediction score}. If the covariance eigenvectors associated with the eigenvalues $\lambda_i\ge\lambda_0$ lie in $\mathcal V$ for some $\lambda _0>0 $, then
\begin{equation}
\label{spectral representation}
\mathcal G_\lambda=\sum_{\lambda_i\ge\lambda_0}
\frac{\langle q_i,c\rangle^2}{\lambda_i+\lambda}.
\end{equation}
Consequently, minimizing the ridge regularized training error is equivalent to maximizing
$\mathcal G_\lambda$. 
\end{theorem}

\begin{proof} 
Unrolling the recursion gives
$x_n=(\beta B)^n x_0+\sum_{k=0}^{n-1}\beta^k B^k A u_{n-1-k}$.
Since $\|\beta B\|<1$, the transient $(\beta B)^n x_0$ decays geometrically and after a washout the empirical moments depend only on the input-driven component, which lies in the block Krylov span of  $(B,A)$, that is, $\operatorname{span}\{B^kA:k\ge0\}$. If $B\mathcal V\subseteq\mathcal V$, then this span is contained in $\mathcal V$, hence $c\in\mathcal V$. Expanding the empirical loss yields
$\frac1{T }\sum_{n=  1}^T (w^\top x_n-u_n)^2=w^\top\Sigma w-2w^\top c+\frac1{T }\sum_{n=  1}^T u_n^2$,
so $J_\lambda(w)=w^\top(\Sigma+\lambda I)w-2w^\top c+\frac1{T }\sum u_n^2$. Completing the square (at $w_\lambda=(\Sigma+\lambda I)^{-1}c$) gives
$J_\lambda(w)=(w-w_\lambda)^\top(\Sigma+\lambda I)(w-w_\lambda)+\frac1{T }\sum u_n^2-c^\top(\Sigma+\lambda I)^{-1}c$.
Since $\Sigma+\lambda I\succ0$, the minimum is attained at $w_\lambda$, yielding the stated identity. The spectral representation \eqref{spectral representation} follows from $\Sigma=\sum_i\lambda_i q_iq_i^\top$.
\end{proof}

\subsection{Standing of our work}
A large body of work studies reservoir performance through {\bfi memory capacity} (MC) \cite{Jaeger:2002},
that is, the ability of a recurrent state to reconstruct past inputs from a linear readout.
Recent analyses \cite{RC15} emphasize that, for commonly used random reservoirs, MC is often governed
by global controllability-type properties of the pair $(B,A)$ (e.g., ranks of associated
controllability matrices in linear settings) and by generic bounds depending on the number of neurons and the input second-order dependence structure in nonlinear settings. More specifically, the forecasting--memory framework of~\cite{RC15} develops general capacity bounds and proves, in the classical linear/i.i.d.\ case, that MC equals the rank of a controllability matrix, stressing the role of input statistics and global expressivity. However, some studies~\cite{RC23, RC34} highlight that random ESN parameterizations tend to yield near-maximal MC under broad conditions, calling into question MC as a discriminative design objective in the linear/random regime. 

Our approach differs in that we do not seek to maximize a global capacity (rank) measure.
Instead, we shape the \emph{geometry of reservoir increments} so that the state cloud effectively becomes
a ``thin affine tube": increments $\Delta x_n$ concentrate in a narrow cone around a
data-defined subspace $\mathcal V:=\mathrm{span}\{v_n\}$ determined by the input-induced
forcing directions $v_n$.
This geometric concentration yields an effectively low-rank covariance spectrum and, crucially,
controls {\bfi covariance leakage} of supervision into $\mathcal V^\perp$, which directly
governs the stability of ridge regression readouts.
In this sense, ``memory'' is not treated as an input-agnostic global invariant of the reservoir; rather, it is structured by enforcing (approximate) invariance of $\mathcal V$ under the feedback dynamics, so that the relevant Krylov chain closes (up to controlled leakage) inside $\mathcal V$ and ridge regression resolves informative directions without amplifying noise.

\paragraph{Organization of the paper.}
The paper proceeds in three conceptual layers, separating geometric structure from the reservoir's linearity or nonlinearity. We first establish that cone concentration
of state increments around a subspace $\mathcal V$ implies small
covariance leakage outside $\mathcal V$.
This result is purely geometric and model-independent
(see Theorem~\ref{thm:approx_covleak_scalar} (already stated)
and Theorem~\ref{thm:approx_covleak} in Section \ref{Cone angle and Spectral Shaping}).

We then prove that when the dominant covariance eigenvectors and the cross-covariance are nearly confined to $\mathcal V$, high-variance directions cannot exhibit vanishing alignment with the supervision (Theorem~\ref{thm:novanish_scalar} (already stated) and Theorem \ref{thm:novanish_vector_robust} in Section \ref{Cone angle and Spectral Shaping}).   Consequently, ridge-relevant variance concentrates in the same subspace that carries most of the covariance structure. This result is also model-independent.

In Section~\ref{Construction of the linear reservoir}, we construct a connectivity matrix $B$ that ``aligns" variance with supervision, yielding minimal ridge training error (Theorems~\ref{thm:subspace_ridge_scalar} and~\ref{thm:subspace_ridge_vector}). Section~\ref{sec: general reservoirs} extends this result to general reservoirs (Theorem~\ref{thm:cone_ridge_general}). Numerical results are presented in Section~\ref{sec: implementation}, followed by discussion in Section~\ref{sec: discussion}.

\section{Cone angle and spectral shaping}
\label{Cone angle and Spectral Shaping}

We now prove the vector-valued input and target versions of Theorems~\ref{thm:approx_covleak_scalar} and~\ref{thm:novanish_scalar}. Throughout, the Euclidean spaces $\mathbb{R}^m$ are equipped with the standard Euclidean inner product and induced $2$-norm. All matrices have real entries. Given a matrix $A$, the symbol $\left\|A\right\|_F $ denotes its Frobenius norm. In the following paragraphs, given a vector subspace $\mathcal V\subset\mathbb{R}^N$, the symbols $P _{\mathcal{V}}, P_{\mathcal{V}^\perp} $ will denote both the orthogonal projections $P _{\mathcal{V}}, P_{\mathcal{V}^\perp} :\mathbb{R}^N \rightarrow \mathbb{R}^N$ and the corresponding matrix expressions $P _{\mathcal{V}}, P_{\mathcal{V}^\perp} \in \mathbb{R}^{N \times N} $. 

\begin{theorem}[Small covariance leakage from cone-concentrated increments]
\label{thm:approx_covleak}
Let $\{x_n\}_{n=0}^T\subset\mathbb{R}^N$ be a sequence of centered states of a system like \eqref{eq:state_space} with
$\Delta x_n=x_{n+1}-x_{n}$ and let
$\{y_n\}_{n=1}^T\subset\mathbb{R}^p$ be centered targets.
Define the state-target cross-covariance by
\[
C:=\frac1T\sum_{n=1}^T x_n y_n^\top.
\]
Fix a $d$-dimensional subspace $\mathcal V\subset\mathbb{R}^N$.
Assume:
\begin{description}
\item [(i)]  \emph{Cone condition on increments}: There exists $\mathcal G\subset\{1,\dots,T\}$, $|\mathcal G|\ge (1-\alpha)T$ and $M_\Delta>0 $ such that 
\[
\|P_{\mathcal V^\perp}\Delta x_n\|\le \tan\theta\,\|P_{\mathcal V}\Delta x_n\|,\qquad n\in\mathcal G,
\]
and $\|\Delta x_n\|\le M_\Delta$ for all $n$.
\item [(ii)]  \emph{Uniform transverse boundedness}: There exists $M_\perp>0$ independent of $T$ with $\|P_{\mathcal V^\perp}x_n\|\le M_\perp$ for $n=1,\dots,T$.
\end{description}
Then
\[
\|P_{\mathcal V^\perp} C\|_F \le M_\perp\,\sigma_y, 
\qquad
\sigma_y:=\Big(\frac1T\sum_{n=1}^T \|y_n\|^2\Big)^{1/2}.
\]
In particular, the covariance leakage ratio $\|P_{\mathcal V^\perp}C\|_F/\|C\|_F$ is small whenever $M_\perp$ is small and $\|C\|_F\neq 0$, independently of $T$.
\end{theorem}

\begin{proof}
By Cauchy--Schwarz applied to the time index,
\[
\|P_{\mathcal V^\perp}C\|_F
=
\Big\|\frac1T\sum_{n=1}^T (P_{\mathcal V^\perp}x_n)\,y_n^\top\Big\|_F
\le
\Big(\frac1T\sum_{n=1}^T \|P_{\mathcal V^\perp}x_n\|^2\Big)^{1/2}\sigma_y
\le
M_\perp\sigma_y. \qedhere
\] Although the proof uses only the transverse boundedness assumption (ii), the increment cone condition in (i) provides a contextual motivation -- the geometric mechanism motivating why such boundedness should hold in reservoir systems. Indeed, if the increments remain persistently concentrated near $\mathcal V$ and the transverse dynamics are stable, then the accumulated transverse component of the trajectory is expected to remain uniformly controlled. 
\end{proof}

\begin{remark} \rm 
    Although the bound above on $\|P_{\mathcal V^\perp}C\|_F$ may grow with $T$ because it controls the
transverse component by a cumulative sum of increments, the quantities entering
the cone condition have a natural stationary interpretation. If the input process
$\{u_n\}$ is stationary then the
corresponding reservoir state process $\{x_n\}$ is stationary after washout, and
so is the increment process $\{\Delta x_n\}$. Moreover, under ergodicity, the
empirical covariance of the forcing increments defining the input-determined
subspace $\mathcal V$ is identical for all typical realizations and thus effectively in the large-sample limit, and the angle $
\angle(\Delta x_n,\mathcal V)$
is itself a stationary observable along such trajectories. Thus cone concentration
can be assessed empirically by estimating the distribution
of these angles and the leakage bound on $\|P_{\mathcal V^\perp}C\|_F$  is a
finite-sample worst-case estimate.
\end{remark}

\begin{theorem}[Robust no-vanishing alignment under cone concentration and small covariance leakage]
\label{thm:novanish_vector_robust}
Let $\Sigma=\frac1T\sum_{n=1}^T x_n x_n^\top$ and
$C=\frac1T\sum_{n=1}^T x_n y_n^\top\in\mathbb{R}^{N\times p}$.
Fix a $d$-dimensional subspace $\mathcal V\subset\mathbb{R}^N$ and define
$C_{\mathcal V^\perp}:=P_{\mathcal V^\perp}C$ and
$C_{\mathcal V}:=P_{\mathcal V}C$.
Assume:
\setlength{\itemsep}{2pt}
\setlength{\parskip}{0pt}
\begin{description}
\item[(A1)] (\emph{Cone concentration of dominant covariance directions})
For some $\varepsilon\in[0,1)$, let $q_i$ be the eigenvector of $\Sigma$ corresponding to the eigenvalue $\lambda _i $ and assume that it satisfies $\|P_{\mathcal V^\perp}q_i\| \le \varepsilon$ (equivalently $\|P_{\mathcal V}q_i\|\ge \sqrt{1-\varepsilon^2}$).
\item[(A2)] (\emph{Nondegenerate supervision within $\mathcal V$})
$\rho := \sigma_{\min}(C_{\mathcal V})>0$.
\item[(A3)] (\emph{Small covariance leakage})
$\delta := \|C_{\mathcal V^\perp}\|$ is small.
\end{description}
Then,
\[
\|q_i^\top C\|
\ge
\sqrt{1-\varepsilon^2}\,\rho - \varepsilon\,\delta.
\]
In particular, if $\varepsilon$ and $\delta$ are small and $\rho$ is not,
then high-variance directions cannot exhibit negligible ridge-relevant alignment $\|q_i^\top C\|$. 
\end{theorem}
\begin{proof}
Fix $q=q_i$ and decompose
$q=q_{\mathcal V}+q_{\mathcal V^\perp}$ with $q_{\mathcal V}=P_{\mathcal V}q$ and
$q_{\mathcal V^\perp}=P_{\mathcal V^\perp}q$. Write also $C=C_{\mathcal V}+C_{\mathcal V^\perp}$ with
$C_{\mathcal V}=P_{\mathcal V}C$ and $C_{\mathcal V^\perp}=P_{\mathcal V^\perp}C$.
Then
\[
q^\top C
=
q_{\mathcal V}^\top C_{\mathcal V}
+
q_{\mathcal V}^\top C_{\mathcal V^\perp}
+
q_{\mathcal V^\perp}^\top C_{\mathcal V}
+
q_{\mathcal V^\perp}^\top C_{\mathcal V^\perp}.
\]
Since the columns of $C_{\mathcal V}$ lie in $\mathcal V$ and $q_{\mathcal V^\perp}\in\mathcal V^\perp$,
we have $q_{\mathcal V^\perp}^\top C_{\mathcal V}=0$.
Likewise, the columns of $C_{\mathcal V^\perp}$ lie in $\mathcal V^\perp$, so $q_{\mathcal V}^\top C_{\mathcal V^\perp}=0$.
Hence
\[
q^\top C = q_{\mathcal V}^\top C_{\mathcal V} + q_{\mathcal V^\perp}^\top C_{\mathcal V^\perp}.
\]
Therefore,
\[
\|q^\top C\|
\ge
\|q_{\mathcal V}^\top C_{\mathcal V}\|
-
\|q_{\mathcal V^\perp}^\top C_{\mathcal V^\perp}\|.
\]
By (A2),
\[
\|q_{\mathcal V}^\top C_{\mathcal V}\|
=\|(C_{\mathcal V})^\top q_{\mathcal V}\|
\ge
\sigma_{\min}(C_{\mathcal V})\,\|q_{\mathcal V}\|
=
\rho\,\|q_{\mathcal V}\|.
\]
By (A3),
\[
\|q_{\mathcal V^\perp}^\top C_{\mathcal V^\perp}\|\le \|q_{\mathcal V^\perp}\|\,\|C_{\mathcal V^\perp}\|
\le \varepsilon\,\delta,
\]
using (A1) to bound $\|q_{\mathcal V^\perp}\|\le\varepsilon$.
Finally, (A1) also gives $\|q_{\mathcal V}\|\ge\sqrt{1-\varepsilon^2}$, so
\[
\|q^\top C\|
\ge
\sqrt{1-\varepsilon^2}\,\rho - \varepsilon\,\delta. \qedhere
\] 
\end{proof}

\begin{remark}[Scalar target]\label{rem:novanish_scalar_robust}
\normalfont
When $p=1$, $C=c\in\mathbb{R}^N$, and $\|q_i^\top C\|_2=|\langle q_i,c\rangle|$, Theorem~\ref{thm:novanish_vector_robust} yields that
\[
|\langle q_i,c\rangle|
\ \ge\
\sqrt{1-\varepsilon^2}\,\|P_{\mathcal V}c\|_2
\;-\;
\varepsilon\,\|P_{\mathcal V^\perp}c\|_2,
\]
for all $i$ with $\lambda_i\ge\lambda_0$.
Thus, if dominant covariance directions lie in a narrow cone around $\mathcal V$
and the covariance leakage $\|P_{\mathcal V^\perp}c\|/\|c\|$ is small,
then large-variance directions cannot have negligible alignment with the target.
\end{remark}

The next proposition enunciates a simple fact that gives an idea of how the cone angle $\theta$ constrains the ratio of the variances in the orthogonal subspaces:

\begin{proposition}[Ratio of transverse to longitudinal variance]
\label{prop:trace-dominance}
Let $P_{\mathcal V}$ and $P_{\mathcal V^\perp}$ denote the orthogonal
projectors onto $\mathcal V$ and $\mathcal V^\perp$, respectively.
If the states satisfy the cone condition $\|P_{\mathcal V^\perp} x_n\|
\le
\tan\theta \, \|P_{\mathcal V} x_n\|
\quad \text{for all } n,$
then
\[
\operatorname{tr}\!\big(
P_{\mathcal V^\perp} \Sigma P_{\mathcal V^\perp}
\big)
\le
\tan^2\theta\,
\operatorname{tr}\!\big(
P_{\mathcal V} \Sigma P_{\mathcal V}
\big).
\]
In particular, if the data are centered, then
\[
\frac{
\text{variance in } \mathcal V^\perp
}{
\text{variance in } \mathcal V
}
\le
\tan^2\theta.
\]
\end{proposition}

\begin{proof}
Since $\Sigma=\frac1T\sum_{n=1}^T x_n x_n^\top$ and orthogonal projectors are symmetric and idempotent,
\[
\operatorname{tr}(P_{\mathcal V^\perp}\Sigma P_{\mathcal V^\perp})
=\frac1T\sum_{n=1}^T \|P_{\mathcal V^\perp}x_n\|^2,
\qquad
\operatorname{tr}(P_{\mathcal V}\Sigma P_{\mathcal V})
=\frac1T\sum_{n=1}^T \|P_{\mathcal V}x_n\|^2.
\]
The cone condition implies
$\|P_{\mathcal V^\perp}x_n\|^2\le\tan^2\theta\,\|P_{\mathcal V}x_n\|^2$
for all $n$, and summing yields the result.
\end{proof}

\section{Construction of data-aligned linear reservoirs (DAR)}
\label{Construction of the linear reservoir}

In this section, we use the insights from the previous results to devise a strategy for constructing linear reservoirs that minimize covariance leakage and avoid vanishing alignment under cone concentration. We call the resulting architectures data-aligned reservoirs (DAR), and we present in a general set up with possible a nonlinear feature map $\phi$.

\paragraph{On the role of the input matrix $A$.} Consider a linear reservoir system for which, this time, we allow the presence of nonlinear input functions $\phi: {\Bbb R}^d \rightarrow \mathbb{R}^d$, that is
\begin{equation}
\label{linear system with phi}
x_{n+1} = A \phi(u_n) + \beta B x_n, \quad \mbox{$0<\beta<1$}.
\end{equation}
The exact choice of the input matrix $A$ is not especially important for the geometric mechanism described above, as long as it does not create strong anisotropy inside the increment spine  $\mathcal{V} = \operatorname{span}\{\Delta v_n \}$ with $v_n = A \phi(u_n)$ (generalized from \eqref{increment-spine} to include $\phi$).
In practical terms, this means that the singular values of $A$, when restricted to the image of $\phi$, should all be of comparable size. When this balanced scaling holds, the variance within $\mathcal V$ remains evenly distributed, and no single direction disproportionately dominates the covariance structure.

Under these conditions, the ridge prediction score introduced in \eqref{error with gain}
\[
\mathcal G_\lambda
=
\operatorname{tr}\!\big(C^\top(\Sigma+\lambda I)^{-1}C\big)
\]
is driven mainly by the geometry of $\mathcal V$ and the feedback design encoded in $B$, rather than by the detailed structure of $A$. Only in pathological cases where $A$ dramatically amplifies one direction while suppressing others does the conditioning significantly deteriorate due to variance collapsing into a single dominant mode. Absent such extreme spectral imbalance, however, the qualitative conclusions of the cone-concentration framework remain intact.

\paragraph{Design of Data-aligned reservoir (DAR) matrix $B$.} 
The results in the previous section prescribe that the state increments cone should be cone-concentrated around $\mathcal V$
and the action of $B$ must have \emph{small transverse component} on $\mathcal V$, that is,
\begin{equation}
\label{approximate invariance condition}
\|P_{\mathcal V^\perp}B v\| \ll \|v\|
\quad \text{for } v\in\mathcal V,
\end{equation}
which constitutes an approximate version of the invariance condition \eqref{invariance condition}. Indeed, a sufficient design condition to satisfy \eqref{approximate invariance condition} is that $B\mathcal V\subseteq\mathcal V$. More generally, $B$ should approximately preserve a PCA approximation of $\mathcal V$, that is, $\mathcal{V}$ is in practice constructed out of a PCA cloud.   We realize the construction of $B$ along these lines in three steps.

\subsection{Step 1: Build local bases and form a shift-alignment problem}
Let $U\in\R^{N\times d}$ have $d$ orthonormal columns spanning $\calV$ (obtained, for instance, from an SVD/PCA decomposition of $\left\{v _n\right\}$).
Define {\bfi coordinates} of the increment forcing within $\calV$:
\begin{equation}
p_n := U^\top v_n \in \R^d.
\label{eq:pn}
\end{equation}
Form two block matrices with time-shifted coordinate snapshots as
\begin{equation}
P_{\mathrm{current}} := [p_0,\dots,p_{T-2}],\qquad
P_{\mathrm{next}} := [p_1,\dots,p_{T-1}].
\label{eq:Pblocks}
\end{equation}

\subsection{Step 2: Orthogonal alignment on $\calV$ (Kabsch)}
We choose an orthogonal map $R_d\in O(d)$ that best aligns $P_{\mathrm{current}}$
to $P_{\mathrm{next}}$:
\begin{equation}
R_d
=
\arg\min_{R\in O(d)} \| P_{\mathrm{next}}- R P_{\mathrm{current}}\|_F^2.
\label{eq:Rd-prob}
\end{equation}
This is the classical Kabsch (or orthogonal Procrustes) problem \cite{Schonemann1966, kabsch1976solution, kabsch1978discussion} that admits the following closed-form solution:
if $P_{\mathrm{next}}P_{\mathrm{current}}^\top = U_s \Sigma_s V_s^\top$ is an SVD,
then $R_d=U_s V_s^\top$. 
We emphasize that $R_d$ is a $d\times d$ matrix even though $v_n\in\R^N$ because the vectors $v_n$ are $N$-vectors, but their informative content is assumed to lie in the $d$-subspace $\calV$.
The matrix $U$ provides an orthonormal basis for $\calV$ for which $p_n=U^\top v_n$
are the $d$-dimensional coordinates. Thus $R_d$ describes the desired linear transformation \emph{on the subspace} $\calV$ prescribed by the Kabsch problem.

\subsection{Step 3: Lift $R_d$ to an ambient feedback matrix}
\label{subsec:lift}
Define the {\bfi aligned} feedback operator on $\R^N$ by acting as $R_d$ on $\calV$ and as the identity on $\calV^\perp$:
\begin{equation}
B_{\mathrm{align}}
:= U R_d U^\top + (I-UU^\top).
\label{eq:Bal}
\end{equation}
Then $B_{\mathrm{align}}$ is orthogonal and satisfies $B_{\mathrm{align}}z\in\calV$
for all $z\in\calV$ (no leakage for vectors already in $\calV$).
The matrix $B_{\mathrm{align}}$ is not adequate as yet, although it prevents the leakage from $\calV$ as explained next. 

The Procrustes step determines the action of $B_{\mathrm{align}}$ on the input-determined
subspace $\mathcal V$, but it does not uniquely specify the action on
$\mathcal V^\perp$. We therefore complete the reservoir matrix by an
orthogonal action on $\mathcal V^\perp$. Choosing this action at random
from the Haar measure introduces no preferred direction outside the
data-determined subspace and prevents artificial coherent accumulation
of transverse components. Thus the random orthogonal completion preserves
the Procrustes alignment on $\mathcal V$ while providing an isotropic,
unbiased mixing mechanism in the unresolved complement.

Define the projected reservoir increment
\[
P_{\mathcal{V}^\perp} \Delta x_{n+1}=\sum_{k=0}^{T} y_k,\qquad 
y_k := \beta^k\,P_{\mathcal{V}^\perp} B^k v_{n+1-k}\in \mathcal{V}^\perp .
\]
\begin{equation}
\label{eq:diag-cross}
\Big\|P_{\mathcal{V}^\perp} \Delta x_{n+1}\Big\|^2
=\Big\|\sum_{k=0}^{T} y_k\Big\|^2
=\sum_{k=0}^{T}\|y_k\|^2
+2\sum_{0\le k<j\le T}\langle y_k,y_j\rangle.
\end{equation}
Equation~\eqref{eq:diag-cross} shows that any growth of 
$\|P_{\mathcal{V}^\perp} \Delta x_{n+1}\|$ beyond the root--sum--of--squares level
is entirely controlled by the cross terms 
$\langle y_k,y_j\rangle$.
If the directions of the vectors $y_k$ are generic in the
$p$-dimensional space $\mathcal{V}^\perp$ ($p=N-d$), then it is possible to prove \cite{Vershynin:book} that pairwise inner products are typically of order $O(p^{-1/2})$  by concentration of measure on the sphere.
Consequently, the cross-term sum in~\eqref{eq:diag-cross} fluctuates
around zero and does not accumulate coherently.
In contrast, when $B=I$, the vectors $y_k=\beta^k P_{\mathcal{V}^\perp} v_{n+1-k}$
remain aligned across delays, making the cross terms strictly positive
and yielding near-coherent accumulation.
A random orthogonal reservoir suppresses this alignment by continually rotating leakage directions in $\mathcal{V}^\perp$, thereby enforcing root--sum--of--squares scaling of $\|P_{\mathcal{V}^\perp} \Delta x_{n+1}\|$ up to
dimension-controlled corrections. Consequently, we blend $B_{\mathrm{align}}$ with an independent orthogonal mixer $B_{\text{rand}}\in O(N)$ constructed by drawing a random orthogonal matrix using the invariant Haar measure on $O(N) $ (available by the compactness of this group \cite[Section 2.1.7]{Ortega2004})
using a {\bfi mixing parameter} $\gamma_{\mathrm{mix}}\in[0,1]$:
\begin{equation*}
\widetilde B
:=(1-\gamma_{\mathrm{mix}})B_{\mathrm{align}}+\gamma_{\mathrm{mix}}B_{\text{rand}}.
\end{equation*}
Finally, we want, for reasons that we explain below, that the final connectivity matrix $B$ be orthogonal; we hence define:
\begin{equation*}
B:=\mathrm{polar}(\widetilde B),
\end{equation*}
where $\mathrm{polar}(\widetilde B)$ denotes the orthogonal factor in the polar decomposition
$\widetilde B=UH$ (that is, the closest orthogonal matrix to $\widetilde B$ in Frobenius norm).

\paragraph{Why the restriction to orthogonal matrices.}
One could, in principle, design the connectivity matrix using arbitrary matrices.
However, the cone concentration of the increments requires quantitative
spectral control of the transverse operator
$B_\perp = P_{\mathcal V^\perp} B P_{\mathcal V^\perp}$.
Writing the increment recursion in transverse coordinates yields
\[
\|P_{\mathcal V^\perp}\Delta x_{n+1}\|
\le
\beta \|B_\perp\|_2
\|P_{\mathcal V^\perp}\Delta x_n\|
+
\beta \|P_{\mathcal V^\perp} B P_{\mathcal V}\|_2
\|P_{\mathcal V}\Delta x_n\|.
\]
Thus, geometric stability requires the spectral condition
$\beta \|B_\perp\|_2 < 1$ together with small cross-leakage
$\|P_{\mathcal V^\perp} B P_{\mathcal V}\|_2$.
For a general matrix $B$ this imposes coupled constraints
on both singular values and invariant directions.

In contrast, when $B$ is orthogonal,
all singular values of $\beta B$ equal $\beta$,
so $\|B_\perp\|_2 = 1$ and the stability condition
reduces simply to $\beta<1$, which is already required for fading memory.
Orthogonality therefore, separates memory scaling (controlled by $\beta$)
from geometric alignment (controlled by subspace invariance),
making cone concentration analytically transparent.

\begin{remark}[On not designing $B$ for exact invariance] \label{rem:block} \normalfont 
One may alternatively construct the reservoir matrix so that the decomposition
$\mathbb R^N=\mathcal V\oplus\mathcal V^\perp$ is exactly invariant under the
reservoir dynamics. In this case,
$B=UR_dU^\top+U_\perp R_\perp U_\perp^\top$, where $R_d$ is obtained from the
Procrustes alignment on $\mathcal V$ and $R_\perp\in O(N-d)$ acts on the
orthogonal complement. Such a construction satisfies
$P_{\mathcal V^\perp}BP_{\mathcal V}=0$, so the transverse leakage is minimized
and the reservoir increments become strongly cone-concentrated around
$\mathcal V$. However, if the forcing itself is low-dimensional, then the
reservoir trajectory may remain confined almost entirely within $\mathcal V$.
In this regime, the effective reservoir dynamics may collapse to approximately
$\dim(\mathcal V)$ degrees of freedom, thereby losing the high-dimensional
embedding mechanism responsible for expressive memory representations.
Numerical experiments not shown here indicate that excessively strong cone concentration can therefore degrade forecasting performance despite improving geometric alignment.

The DAR construction intentionally avoids exact invariance. After the mixing and
polar projection step, the resulting orthogonal matrix no longer preserves the
decomposition $\mathcal V\oplus\mathcal V^\perp$, and one typically has
$P_{\mathcal V^\perp}BP_{\mathcal V}\neq 0$. Consequently, the reservoir states
are allowed to explore ambient directions outside $\mathcal V$, producing a
richer nonlinear embedding while still maintaining approximate cone
concentration. For this reason, one should not expect a clean invariant-subspace
characterization of the reservoir matrix. The geometry is instead controlled
heuristically through approximate invariance and small transverse leakage rather
than through an exact block-diagonal decomposition.
\end{remark}

\subsection{Training error with a designed connectivity matrix $B$}
\label{subsec:training_error}
\begin{lemma}[Mixing parameter induces controlled transverse leakage]
\label{lem:mixing-leakage}
Let $B$ be constructed via
\[
\widetilde B
=(1-\gamma_{\mathrm{mix}})B_{\mathrm{align}}
+\gamma_{\mathrm{mix}}B_{\mathrm{rand}},
\qquad
B=\operatorname{polar}(\widetilde B),
\]
where $B_{\mathrm{align}}$ and $B_{\mathrm{rand}}$ are orthogonal matrices and
$B_{\mathrm{align}}\mathcal V\subseteq\mathcal V$.
Then there exists a constant $K>0$, independent of $\gamma_{\mathrm{mix}}$,
such that for all $v\in\mathcal V$,
\[
\|P_{\mathcal V^\perp} B v\|
\le
K\,\gamma_{\mathrm{mix}}\,\|v\|.
\]
In particular, approximate invariance holds with
$\gamma = K\,\gamma_{\mathrm{mix}}$.
\end{lemma}

\begin{proof}
Let $v\in\mathcal V$. Since $B_{\mathrm{align}}\mathcal V\subseteq\mathcal V$, 
$P_{\mathcal V^\perp}B_{\mathrm{align}}v=0$, hence
\[
P_{\mathcal V^\perp}\widetilde B v
=\gamma_{\mathrm{mix}}P_{\mathcal V^\perp}B_{\mathrm{rand}}v,
\quad
\|P_{\mathcal V^\perp}\widetilde B v\|
\le \gamma_{\mathrm{mix}}\|v\|.
\tag{1}
\]
Write the polar decomposition $\widetilde B=BJ$ with 
$J=(\widetilde B^\top\widetilde B)^{1/2}$. Since $B_{\mathrm{align}}$ and 
$B_{\mathrm{rand}}$ are orthogonal,
\[
\widetilde B
= B_{\mathrm{align}}
+\gamma_{\mathrm{mix}}(B_{\mathrm{rand}}-B_{\mathrm{align}}),
\]
so $\|\widetilde B-B_{\mathrm{align}}\|\le 2\gamma_{\mathrm{mix}}$, hence
$\|\widetilde B^\top\widetilde B-I\|=O(\gamma_{\mathrm{mix}})$ and 
$\|J-I\|=O(\gamma_{\mathrm{mix}})$. Since $B=\widetilde B J^{-1}$ and 
$J^{-1}=I+O(\gamma_{\mathrm{mix}})$,
\[
\|B-\widetilde B\|=O(\gamma_{\mathrm{mix}}).
\tag{2}
\]
Combining (1)–(2),
\[
\|P_{\mathcal V^\perp}Bv\|
\le
\|P_{\mathcal V^\perp}\widetilde Bv\|
+
\|(B-\widetilde B)v\|
\le
K\gamma_{\mathrm{mix}}\|v\|.
\qedhere\]
\end{proof}

Theorem~\ref{thm:subspace_ridge_vector} below is an extension of Theorem~\ref{thm:subspace_ridge_scalar} to a vector target while also taking into account the leakage $\gamma$. The reader should note that Theorem~\ref{thm:subspace_ridge_vector} does not imply that increasing $\gamma$ improves prediction performance. While smaller $\gamma$ improves geometric alignment by reducing covariance leakage, excessively small $\gamma$ may reduce the effective embedding dimension of the reservoir. Here the ridge prediction score reflects a tradeoff between covariance concentration and high-dimensional memory representation.

\begin{theorem}[Subspace invariance and ridge training error, vector case]
\label{thm:subspace_ridge_vector}
Consider the linear reservoir
$
x_{n+1}=A u_n+\beta B x_n$, $0<\beta<1$, with $B$ orthogonal and
centered inputs $u_n\in\mathbb{R}^p$. Define
\[
\Sigma=\tfrac1T\sum_{n=1}^T x_n x_n^\top, 
\quad
C=\tfrac1T\sum_{n=1}^T x_n u_n^\top, \] 
and let $\mathcal V\subset\mathbb{R}^N$ be the increment-spine which is $\operatorname{span}\{\Delta v_n \}$ with $v_n = A u_n$.  
Assume approximate invariance
\[
\|P_{\mathcal V^\perp}B v\|\le \gamma \|v\|,
\qquad v\in\mathcal V,
\]
where $\gamma=K\gamma_{\mathrm{mix}}$. Then:
\begin{description}
\item[(i)] {\bf Exact invariance.} If $\gamma=0$, then $C\in\mathcal V$ (meaning that the range of $C$ is in $\mathcal V$). 

\item[(ii)] {\bf Controlled covariance leakage:} There exists $K_1>0$ independent of $T$ such that
\[
\|C_{\mathcal V^\perp}\|_F \le K_1 \gamma,
\qquad
C_{\mathcal V^\perp}:=P_{\mathcal V^\perp}C.
\]

\item[(iii)] {\bf Ridge identity.} For $\lambda>0$, the ridge minimizer 
$W_\lambda=(\Sigma+\lambda I)^{-1}C$ satisfies
\[
\min_W\Big(
\tfrac1T\sum_{n=1}^T \|W^\top x_n-u_n\|^2
+\lambda\|W\|_F^2
\Big)
=
\tfrac1T\sum_{n=1}^T \|u_n\|^2
-
\mathcal G_\lambda,
\]
where
$
\mathcal G_\lambda
=\operatorname{tr}\!\big(C^\top(\Sigma+\lambda I)^{-1}C\big)$ is the ridge prediction score.

\item[(iv)] {\bf Dominant variance contribution.} Let $\Sigma=\sum_i\lambda_i q_i q_i^\top$, and assume
\[
\|P_{\mathcal V^\perp}q_i\|
\le K_2\gamma
\quad \text{for all }\lambda_i\ge\lambda_0.
\]
Then, the ridge prediction score is such that 
\[
\mathcal G_\lambda
=
\sum_{\lambda_i\ge\lambda_0}
\frac{\|q_{i,\mathcal V}^\top C_{\mathcal V}\|_2^2}{\lambda_i+\lambda}
+
O(\gamma^2),
\]
where $q_{i,\mathcal V}=P_{\mathcal V}q_i$ and
$C_{\mathcal V}=P_{\mathcal V}C$. 
\end{description} 
\end{theorem}

\begin{proof}
By Lemma~\ref{lem:mixing-leakage}, the mixing construction
\[
B=\operatorname{polar}\!\big((1-\gamma_{\mathrm{mix}})B_{\mathrm{align}}
+\gamma_{\mathrm{mix}}B_{\mathrm{rand}}\big)
\]
implies
\[
\|P_{\mathcal V^\perp}B v\|
\le K\gamma_{\mathrm{mix}}\|v\|
\quad \text{for all } v\in\mathcal V.
\]
Set $\gamma:=K\gamma_{\mathrm{mix}}$. We now analyze the four claims in the statement.

\medskip

\noindent\textbf{(i)}
If $\gamma=0$, then $B\mathcal V\subseteq\mathcal V$. Since
$x_{n+1}=A u_n+\beta B x_n$ and $A u_n\in\mathcal V$, induction gives
$x_n\in\mathcal V$ (after transients), hence $C=\tfrac1T\sum_{n=1}^T x_n u_n^\top\in\mathcal V.$

\noindent\textbf{(ii)}
Using $x_n=\sum_{k\ge0}\beta^k B^k A u_{n-k},$
each application of $B$ injects at most $\gamma$ transverse variance, so a geometric series bound yields
\[
\|P_{\mathcal V^\perp}x_n\|\le K_1\gamma,
\]
with $K_1$ independent of $T$. Thus
\[
\|C_{\mathcal V^\perp}\|_F
=\left\|\tfrac1T\sum_{n=1}^T P_{\mathcal V^\perp}x_n\,u_n^\top\right\|_F
\le K_1\gamma.
\]

\noindent\textbf{(iii) }
Expanding,
\[
J_\lambda(W)
=\operatorname{tr}\!\big(W^\top(\Sigma+\lambda I)W\big)
-2\operatorname{tr}(W^\top C)
+\tfrac1T\sum\|u_n\|^2.
\]
Completing the square at $W_\lambda=(\Sigma+\lambda I)^{-1}C$ gives
\[
\min_W J_\lambda(W)
=\tfrac1T\sum\|u_n\|^2
-\operatorname{tr}\!\big(C^\top(\Sigma+\lambda I)^{-1}C\big).
\]

\noindent\textbf{(iv) }
Write
\[
C=C_{\mathcal V}+C_{\mathcal V^\perp},
\qquad
q_i=q_{i,\mathcal V}+q_{i,\perp}.
\]
Since the columns of $C_{\mathcal V}$ lie in $\mathcal V$ and
$q_{i,\perp}\in\mathcal V^\perp$, we have
$q_{i,\perp}^\top C_{\mathcal V}=0$.
Likewise, since the columns of $C_{\mathcal V^\perp}$ lie in $\mathcal V^\perp$
and $q_{i,\mathcal V}\in\mathcal V$, we have $q_{i,\mathcal V}^\top C_{\mathcal V^\perp}=0$.
Hence $q_i^\top C
=
q_{i,\mathcal V}^\top C_{\mathcal V}
+
q_{i,\perp}^\top C_{\mathcal V^\perp}$. From {\bf (ii)}, $\|C_{\mathcal V^\perp}\|_2\le K_1\gamma$, and by hypothesis
$\|q_{i,\perp}\|_2\le K_2\gamma$. Therefore
\[
\|q_i^\top C-q_{i,\mathcal V}^\top C_{\mathcal V}\|_2
=
\|q_{i,\perp}^\top C_{\mathcal V^\perp}\|_2
\le
\|q_{i,\perp}\|_2\,\|C_{\mathcal V^\perp}\|_2
\le
K_1K_2\,\gamma^2 .
\]
It follows that
\[
\|q_i^\top C\|_2^2
=
\|q_{i,\mathcal V}^\top C_{\mathcal V}\|_2^2
+
O(\gamma^2).
\]
Substituting, the ridge prediction score satisfies $\mathcal G_\lambda
=
\sum_{\lambda_i\ge\lambda_0}
\frac{\|q_{i,\mathcal V}^\top C_{\mathcal V}\|_2^2}{\lambda_i+\lambda}
+
O(\gamma^2).$
\end{proof}

\begin{remark}
\normalfont
For a linear reservoir with the ESP, $x_{n}= \sum_{k\ge 0} \beta^k B^k v_{n-k}$, where $v_n=Au_n$ and let $\mathcal V=\operatorname{span}\{\Delta v_n\}$ with $v_n\in\mathcal V$ and $\sup_n\|v_n\|\le M_v<\infty$.  Since $B$ satisfies the design condition $\|P_{\mathcal V^\perp}Bv\|\le\gamma\|v\|$ for all $v\in\mathcal V$, then
$
\|P_{\mathcal V^\perp}x_n\| \le \frac{\beta\gamma\,M_v}{1-\beta}=:M_\perp$, for all $n$.
In particular, $M_\perp\to 0$ as $\gamma\to 0$, providing hypotheses {\bf (ii)} in Theorems \ref{thm:approx_covleak_scalar} and \ref{thm:approx_covleak} with explicit dependence on the design parameter. Later on, when we consider a $\tanh$-ESN (see Corollary~\ref{cor:tanh_training}), these hypotheses become more permissive than in the linear reservoir case. \end{remark}

\begin{remark}[Heuristic near-cone invariance and the role of the input]
\normalfont
Consider the increment dynamics $\Delta x_{n+1}=v_n+\beta B\Delta x_n$. If $B$ approximately preserves $\mathcal V$ (that is, induces only small transverse leakage), then the loss of cone invariance is primarily due to insufficient alignment of $v_n$ with $\mathcal V$. Heuristically, if $\|v_n\|\gtrsim \mu\,\beta\,\|P\Delta x_n\|$, the longitudinal component dominates the leakage, yielding approximate forward invariance of a cone around $\mathcal V$ and keeping increments concentrated near $\mathcal V$. A rigorous treatment of this observation requires additional assumptions on $B$ and $v_n$, which we do not pursue here. Instead, we design $B$ to nearly preserve $\mathcal V$, which suffices for the training error reduction observed in this work. A systematic study of conditions such as $\|v_n\|\gtrsim \mu\,\beta\,\|P\Delta x_n\|$ and their role in ensuring near-cone invariance is left for future work.
\end{remark}

\section{Extension to general reservoirs} 
\label{sec: general reservoirs}

\begin{theorem}[Cone concentration yields non-vanishing ridge gain]
\label{thm:cone_ridge_general}
Assume the hypotheses of Theorems~\ref{thm:approx_covleak}. Let $\epsilon >0  $ and $\lambda_0 >0 $  be such that for all the eigenvectors $\{q_i:\lambda_i\ge\lambda_0\}$ one has that
$\|P_{\mathcal V^\perp}q_i\|\le\varepsilon$. Let
$\rho:=\sigma_{\min}(C_{\mathcal V})>0$ and
$\delta:=\|P_{\mathcal V^\perp}C\|$. Then, the ridge prediction score
$\mathcal G_\lambda=\operatorname{tr}\!\big(C^\top(\Sigma+\lambda I)^{-1}C\big)$ satisfies 
\[
\mathcal G_\lambda\ge
\sum_{\lambda_i\ge\lambda_0}
\frac{(\sqrt{1-\varepsilon^2}\rho-\varepsilon\delta)^2}{\lambda_i+\lambda}.
\]
In particular, if $\sqrt{1-\varepsilon^2}\rho>\varepsilon\delta$, then $\mathcal G_\lambda$ is bounded away from zero, and the ridge training error is strictly reduced by a non-vanishing amount.
\end{theorem}

\begin{proof}
Using the ridge identity $E_\lambda=\tfrac1T\sum\|y_n\|^2-\mathcal G_\lambda,
\qquad
\mathcal G_\lambda=\operatorname{tr}\!\big(C^\top(\Sigma+\lambda I)^{-1}C\big),$ let $\Sigma=\sum_i\lambda_i q_iq_i^\top$. Then
\[
\mathcal G_\lambda=\sum_i\frac{\|q_i^\top C\|^2}{\lambda_i+\lambda}.
\]
By Theorem~\ref{thm:novanish_vector_robust},
\[
\|q_i^\top C\|\ge\sqrt{1-\varepsilon^2}\rho-\varepsilon\delta
\quad \text{for all } \lambda_i\ge\lambda_0.
\]
Therefore
\[
\mathcal G_\lambda\ge
\sum_{\lambda_i\ge\lambda_0}
\frac{(\sqrt{1-\varepsilon^2}\rho-\varepsilon\delta)^2}{\lambda_i+\lambda}.
\]
If $\sqrt{1-\varepsilon^2}\rho>\varepsilon\delta$, then each term in the sum is strictly positive, so $\mathcal G_\lambda$ admits a strictly positive lower bound and is bounded away from zero. Substituting into the ridge identity yields a non-vanishing reduction of the training error.
\end{proof}

The preceding theorem is purely geometric and does not depend on a specific reservoir dynamics. We now show, via a corollary, that common nonlinear reservoirs satisfy the required cone concentration property.

\begin{corollary}[Echo state networks with feature forcing]
\label{cor:tanh_training}
Consider the reservoir
\begin{equation*}
x_{n+1}=\tanh(A\phi(u_n)+Bx_n),
\qquad
v_n=A(\phi(u_{n+1})-\phi(u_n)),
\end{equation*}
and define $\mathcal V=\operatorname{span}\{v_n\}$.
If
\[
\|P_{\mathcal V^\perp}Bv\|\le\gamma\|v\| \ (v\in\mathcal V),\qquad
\|P_{\mathcal V^\perp}Bw\|\le\kappa\|w\| \ (w\in\mathcal V^\perp),
\]
with $\kappa<1$, then the increments $\Delta x_n$
remain concentrated in a cone around $\mathcal V$.
Consequently
Theorems~\ref{thm:approx_covleak},
\ref{thm:novanish_vector_robust}, and
\ref{thm:cone_ridge_general} apply, implying reduced ridge training error
whenever $\gamma$ is small.
\end{corollary}

\paragraph{Interpretation.}
Theorem~\ref{thm:novanish_vector_robust} shows that robust alignment of
dominant covariance directions follows from two geometric properties:
(i) dominant covariance eigenvectors remain close to the increment
subspace $\mathcal V$, and (ii) covariance leakage outside $\mathcal V$
is small.  For nonlinear reservoirs it therefore suffices to verify that
state increments remain concentrated in a cone around $\mathcal V$,
since the covariance geometry is induced by these increments.

For the reservoir matrix constructed in
Section~\ref{subsec:lift}, Lemma~\ref{lem:mixing-leakage} yields $\|P_{\mathcal V^\perp}Bv\|
\le K\,\gamma_{\mathrm{mix}}\|v\|$, when $ v\in\mathcal V$. Hence decreasing the mixing parameter $\gamma_{\mathrm{mix}}$
reduces transverse leakage and sharpens the cone concentration,
thereby improving the alignment between the covariance geometry
and the supervised signal.

\section{DAR implementation and numerical results } \label{sec: implementation}

We generate a scalar input signal using the Lorenz system
\[
\dot x = \sigma (y-x), \qquad
\dot y = x(\rho-z)-y, \qquad
\dot z = xy-\beta_{\mathrm{Lor}} z,
\]
with parameters $\sigma = 10, \qquad \rho = 28, \qquad \beta_{\mathrm{Lor}} = \tfrac{8}{3}$ (see the Supplement for inputs from other systems, namely the Mackey-Glass and the full logistic maps).
The system is discretized using the classical fourth-order Runge--Kutta method with time step $\Delta t = 0.01$, over a trajectory of total length \(T=16000\). The input is the centered first Lorenz coordinate, $u_t = x_t - \frac{1}{T}\sum_{k=1}^T x_k.$

We use reservoirs of two types. On the one hand we consider linear reservoirs ($\phi(u) = u$ in \eqref{linear system with phi}) with a trained neural network readout, feedback parameter ($\beta$ in \eqref{linear system without phi}) set to $\beta=0.9785$, and an input matrix $A \in \mathbb{R}^{N\times 1}$ constructed using independent random entries drawn out of a standard Gaussian distribution. On the other hand, we consider echo state networks as in \eqref{eq:esn_update} of the same dimension, with a linear readout trained using ridge regression.

In the data preparation, a washout period of $1000$ time steps is discarded. The remaining samples $t=1001,\dots,15000$ yield \(14000\) usable data points, which are split into $9800 \text{ training}$,  $2100 \text{ validation}$,  and $2100 \text{ test samples}$. 

To construct the designed reservoir matrix, we define the increment forcing $v_n = A\,(u_{n+1}-u_n)$, and estimate the increment subspace $\mathcal V = \operatorname{span}\{v_n\}$ by principal component analysis (PCA) on the increment cloud \(\{v_n\}\), retaining the smallest number \(d\) of principal directions capturing \(95\%\) of the variance (in the case of a scalar input, this turns out to be only one direction).

The Procrustes alignment is performed on projected increment blocks of depth $K=8$, so that each block contains \(K-1=7\) consecutive increments. A total of \(100\) window centers are used to form the stacked Procrustes matrices to find $R_d$. The resulting subspace operator is lifted to \(\mathbb{R}^N\) and mixed with a random orthogonal matrix using $\gamma_{\mathrm{mix}}=0.15$  before it is projected onto the orthogonal group via the polar decomposition to obtain \(B\in O(N)\) as per the DAR-design.

Two readouts are trained using the same training set of $9450$ samples. The neural-network (NN) readout corresponding to the \emph{linear} dynamics $x_{t+1} = A u_t + \beta B x_t$ is implemented using a
one-hidden-layer network with width \(H=64\) is trained for \(600\) epochs using Adam with learning rate \(3\times 10^{-3}\) and minibatch size \(256\).
The ridge readout is implemented for the echo state network reservoir dynamics $x_{t+1} = \tanh(A u_t + \beta B x_t)$ using a regularization parameter $\lambda_{\mathrm{ridge}} = 10^{-2}$. Thus, both readouts use identical training, validation, and test lengths, but are trained on different state trajectories (linear versus nonlinear reservoir dynamics).  

The forecasting results for a single trial using a reservoir of dimension $N=100$ are shown in panel (a) of Fig.~\ref{fig:dar}, comparing the performance obtained with a random $B$ to that with a designed one. In the top two panels, autonomous forecasts of the first $400$ prediction steps with both the ridge regression and the NN readouts are shown. The reader may note that the mean squared error (MSE) shown in the caption is computed without normalizing the signal amplitudes.  In the bottom panel, the phase portraits along with the Hausdorff distances computed over 2000 time-steps are shown. Fig.~\ref{fig:dar}(b) shows how the cone-aperture is distributed for the same trial, indicating that a designed $B$ keeps the state-increments $\Delta x_n$ closer to the spine $\mathcal V$ more often than a random $B$.  The experiments are repeated $500$ times for reservoir sizes ranging from $N=100$ to $N=1000$ in steps of $100$, where orthonormal matrices are chosen both for a random $B$ and for mixing in a designed $B$,  and the resulting MSE is shown in Fig~\ref{fig:dar}(c).  These experiments illustrate the remarkable reduction in the variability in the performance of the DAR-designed $B$. 

Behind these results, we highlight two distinct but interacting mechanisms governing performance: a \emph{representation problem} and a \emph{dynamical problem}. The representation problem concerns the spectral structure of the state covariance and its alignment with the state--target cross-covariance, as quantified by the ridge decomposition modes $\frac{\langle q_i,c\rangle^2}{\lambda_i+\lambda}$ (see the supplement) whose diffuse contributions across modes indicate spectral pollution. For the Lorenz example, this is usually achieved by both the random and designed reservoirs in this example.  The cone concentration although separately provides a geometric mechanism to reducing covariance leakage and promoting alignment (Theorem~\ref{thm:novanish_vector_robust}).  However, preventing the spectral pollution is not sufficient in general for reliable prediction. Even when the ridge spectrum is typically well-structured for random reservoirs, yet autonomous prediction may still fail -- this is where the structure in the DAR-reservoir $B$ is based on the Procrustes based learning that helps solving the dynamical problem.

We also point out that spectral pollution may persist despite cone concentration even under a designed $B$ due to insufficient observability, for instance in the logistic map (see Supplement). One can experiment to improve observability by trying new features in $\phi$ to reduce the spectral pollution, a topic which needs investigation in a future enterprising work.

\begin{figure}[h]
\label{fig:dar}
\centering
   \includegraphics[width=16cm]{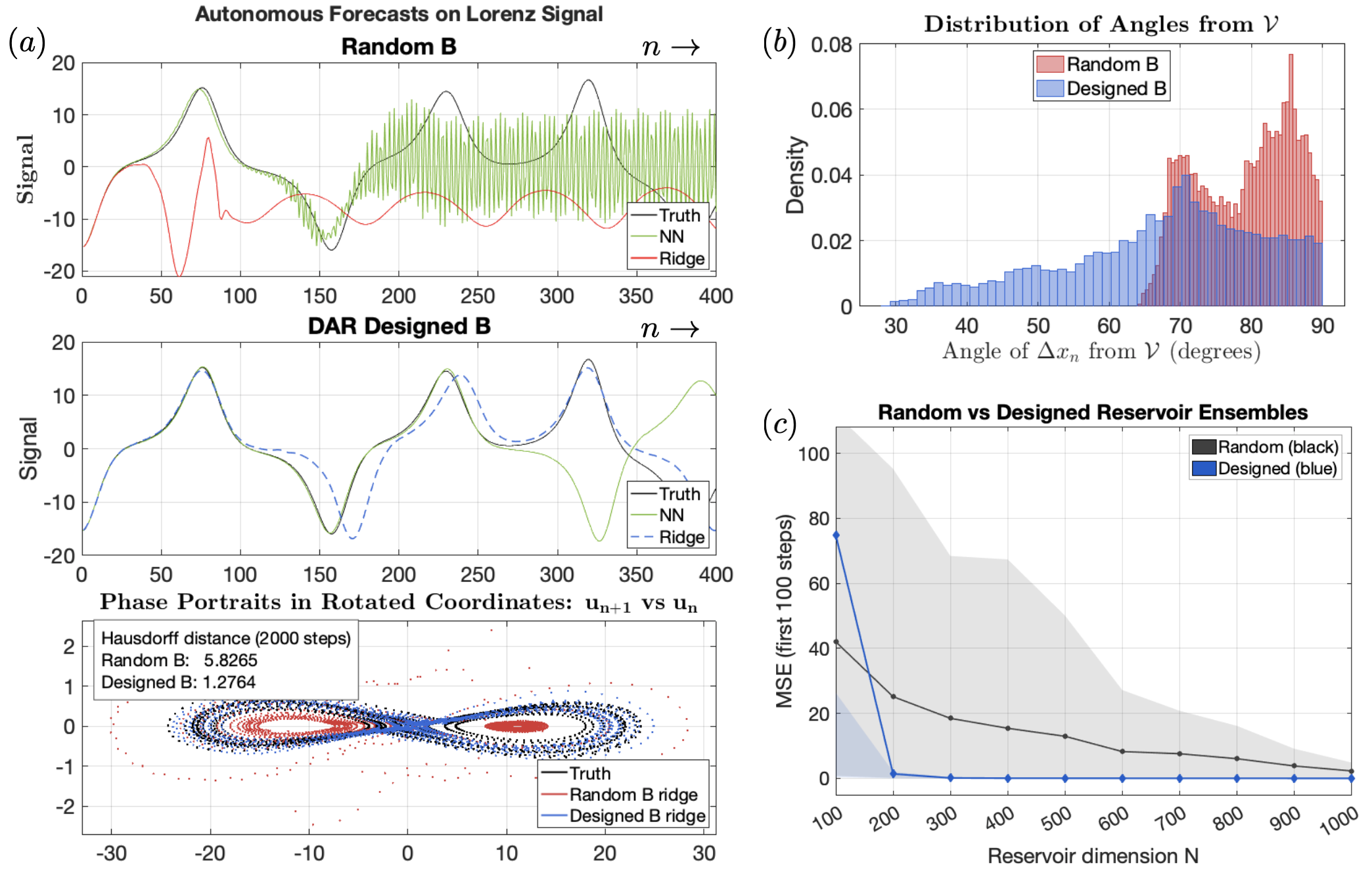}
   \caption{Performance illustration: For the trials in {\bf(a)} and {\bf(b)} reservoirs of size $N=100$ are used. For the DAR designed $B$, MSE in the first $100$ steps: 0.029979 (using neural network readout) and 0.069852 (ridge readout). In {\bf(c)}, shaded bands indicate the 10th--90th percentile range across trials (central $80\%$ variability), while solid lines show the mean error.}
\end{figure}

\section{Discussion and conclusions} \label{sec: discussion}

\subsection{Why a large reservoir dimension $N$ is needed without design?} 

In a linear reservoir of the form \eqref{eq:esn_update} with scalar inputs and whose connectivity matrix $B$ has spectral radius smaller than $1/ \beta$ (see \cite[Proposition 4.2]{RC16}), the solution states admit the explicit representation $ x_n = \sum_{k\ge 0} \beta^k B^k A u_{n-k}$ and hence one can define a linear mapping
\[
h:\ (\ldots,u_{n-2},u_{n-1}) \;\longmapsto\; x_n,
\]
from the (infinite-dimensional) input history space into $\mathbb{R}^N$.
In practice, because the weights $\beta^k$ decay exponentially, only a finite window of past inputs contributes appreciably, so $h$ may be regarded as a linear map from a finite-dimensional Euclidean space $\mathbb{R}^{mL}$ (with $L$ determined by $\beta$) into $\mathbb{R}^N$ (see \cite{RC15, RC23} for details on this observation). Further, the image of $h$ is contained in the Krylov subspace given by
\[
\mathcal R \;=\; \mathrm{span}\big\{ B^k A : \ k\in  \left\{0, \ldots, N-1\right\} \big\},
\]
which is the Kalman controllability subspace associated with the pair $(B, A)$.
Thus the geometry of the reservoir states $\{x_n\}$ is entirely determined by how the linear map
\[
(u_{n-L},\dots,u_{n-1}) \;\longmapsto\; \sum_{k=0}^{L-1}\beta^k B^k A u_{n-k}
\]
embeds this input-history space into $\mathbb{R}^N$. Also, the map $h$ acts approximately as a random linear embedding of a high-dimensional input-history space into $\mathbb{R}^N$. In this regime, unless $N$ is very large compared to the effective dimension of the history space, the images $x_n$ tend to concentrate, leading to near-isotropic but weakly informative covariances. This is precisely the setting in which Johnson--Lindenstrauss type results \cite{JLlemma} indicate that a large ambient dimension $N$ is required to avoid distortion, collapse, or loss of separation of distinct inputs. Separately,  various stability, mixing, or ergodicity hypotheses have also been used recently to explain why density-based forecasting works for large $N$ \cite{RC35}.

\subsection{Conclusions}
Our construction yields low-dimensional, data-specific reservoirs by combining:
(i) an increment subspace $\calV$ derived from lifted feature increments,
(ii) cone-controlled leakage bounds ensuring stable geometry, and
(iii) an orthogonal alignment design of $B$.

\paragraph{Implications for other reservoir platforms.}
Although our construction is formulated in terms of a linear feedback matrix,
the underlying design principle is geometric and platform-independent.
In optical, photonic, or oscillator-based reservoirs,
one does not directly choose matrix entries,
but rather coupling strengths, gain parameters,
or network topologies.
These determine an effective Jacobian operator
that plays the role of $\beta B$.
Our results suggest that such systems could be engineered so that
instead of maximizing generic memory capacity, the geometry of the reservoir should be shaped to have a structured invariant subspace aligned with the input.

\medskip

\noindent {\bf Acknowledgements.} The authors acknowledge fruitful discussion with Lyudmila Grigoryeva. GM gratefully acknowledges generous support for research visits from the School of Physical and Mathematical Sciences (SPMS) of the Nanyang Technological University, Singapore.

\bibliographystyle{wmaainf}
\let\oldthebibliography\thebibliography
\renewcommand{\thebibliography}[1]{%
  \oldthebibliography{#1}%
  \setlength{\itemsep}{0pt}%
}
\bibliography{JPLibrary}

\newpage
\setcounter{page}{1}

\section{Supplement: Additional Numerical Results}

\subsection{Lorenz-63 system}
\paragraph{Histogram of ridge regression modes.}
We plot (see Fig.~\ref{fig:ridge_score}) the normalized empirical distribution of the ridge-mode contributions
\[
\frac{\langle q_i,c\rangle^2}{\lambda_i+\lambda},
\]
where $\lambda_i$ and $q_i$ are the eigenvalues and eigenvectors of the empirical reservoir-state covariance matrix $\Sigma$, $c$ is the state--target cross-covariance vector, and $\lambda>0$ is the ridge regularization parameter. These quantities are computed from the standardized teacher-forced $\tanh$ reservoir states used to train the ridge readout. The histograms are normalized to have unit area, so that the vertical axis represents a probability density.

A concentration of mass at larger ridge-mode values indicates that directions of large state variance are also strongly aligned with the prediction target. Hence, the histogram visualizes how the ridge prediction score is distributed across covariance eigen-directions for the two reservoir constructions and diagnoses spectral pollution. Clearly, the figure indicates that for the Lorenz system, the ridge-mode values do not have such an issue. 

\begin{figure}[h]
\label{fig:ridge_score}
\centering
   \includegraphics[width=8cm]{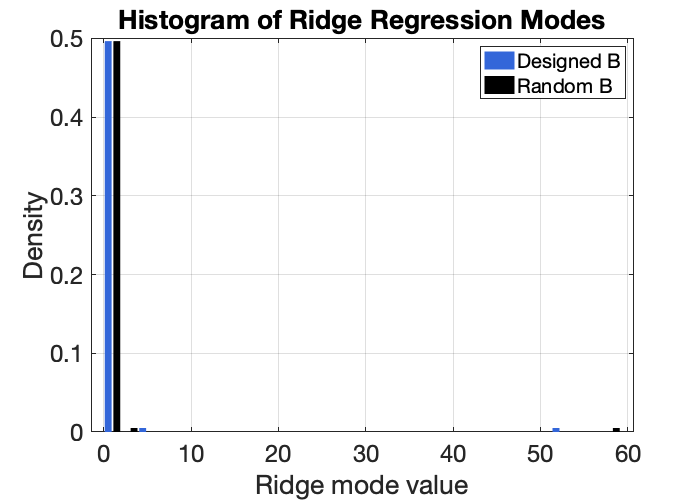}
 \caption{Normalized histogram of ridge regression mode contributions 
$\langle q_i,c\rangle^2/(\lambda_i+\lambda)$, where $\{\lambda_i,q_i\}$ are the eigenpairs of the empirical state covariance matrix $\Sigma$ and $c$ is the state--target cross-covariance vector. The blue bars correspond to the data-aligned (designed) reservoir matrix $B$, while the black bars correspond to a random orthogonal $B$.}
\end{figure}

\subsection{Mackey-Glass system}

The input time series is generated from the Mackey--Glass delay differential equation
\[
\dot{x}(t) = \beta \frac{x(t-\tau)}{1 + x(t-\tau)^n} - \alpha x(t),
\]
with standard parameter choices $\beta = 0.2$, $\alpha = 0.1$, $n = 10$, and delay $\tau = 17$, for which the system exhibits chaotic dynamics. The equation is numerically integrated using a fixed time step, and a scalar time series $\{x(t_k)\}$ is obtained by sampling the solution at discrete times $t_k = k\Delta t$. Following common practice, the initial transient is discarded and the remaining trajectory is treated as a stationary signal.

The resulting sequence is centered by subtracting its empirical mean and used as the scalar input $u_t$ to the reservoir. This dataset, introduced in (Cite Jaeger's science paper), is a standard benchmark for evaluating nonlinear time-series prediction methods due to its long-term chaotic behavior and sensitive dependence on past states.
 
All design settings for the experiment are identical to the Lorenz system. The results are presented in Fig.~\ref{fig:MG_DAR} while the ridge-mode values are plotted in Fig.~\ref{fig:ridge_score_mg}.

\begin{figure}[h]
\label{fig:MG_DAR}
\centering
   \includegraphics[width=16cm]{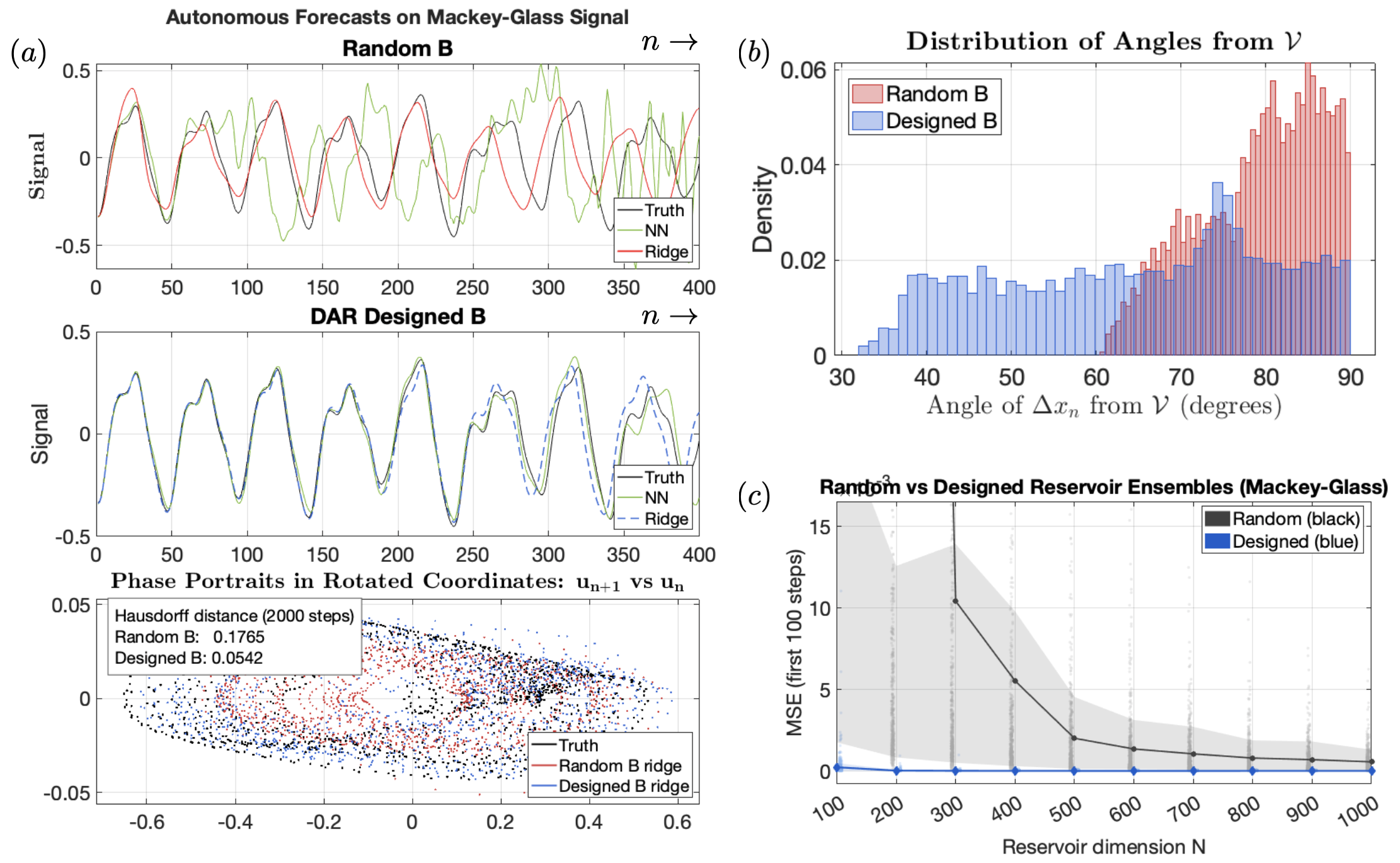}
   \caption{Performance illustration: For the trials in {\bf(a)} and {\bf(b)} reservoirs of size $N=100$ are used. For the DAR designed $B$, MSE in the first $100$ steps for the ridge regression readout: 0.005720 (for Random $B$) and 0.000110 (ridge regression training). In {\bf(c)}, shaded bands indicate the 10th--90th percentile range across trials (central $80\%$ variability), while solid lines show the mean error.}
\end{figure}

\begin{figure}[h]
\label{fig:ridge_score_mg}
\centering
   \includegraphics[width=8cm]{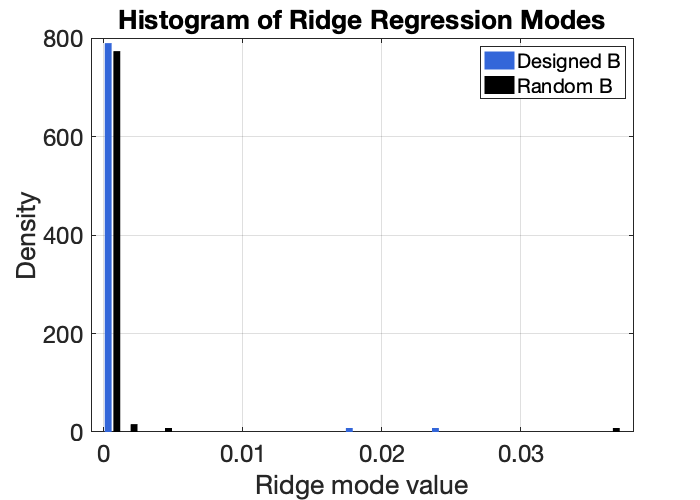}
 \caption{Mackey-Glass System: Normalized histogram of ridge regression mode contributions 
$\langle q_i,c\rangle^2/(\lambda_i+\lambda)$, where $\{\lambda_i,q_i\}$ are the eigenpairs of the empirical state covariance matrix $\Sigma$ and $c$ is the state--target cross-covariance vector. The blue bars correspond to the data-aligned (designed) reservoir matrix $B$, while the black bars correspond to a random orthogonal $B$.}
\end{figure}

\subsection{The full logistic map}

\paragraph{Logistic map data.}
The input time series is generated by the fully chaotic logistic map
\[
u_{t+1} = 4u_t(1 - u_t),
\]
with initial condition $u_0 \in (0,1)$. For this parameter value, the map exhibits chaotic dynamics and admits an invariant measure supported on $[0,1]$. After discarding an initial transient to remove dependence on the initial condition, the resulting trajectory $\{u_t\}$ is treated as a stationary scalar time series.

The sequence is centered by subtracting its empirical mean, yielding the input signal $u_t \leftarrow u_t - \bar{u}$ used to drive the reservoir. 

\paragraph{Effect of spectral pollution.}
For the logistic map input, both the randomly generated reservoir matrix $B$ and the DAR-designed matrix fail to produce accurate predictions  for the ridge regression training (for sanity results are not shown here) although the neural network training  performs very well. This degradation in performance can be attributed to spectral pollution in the reservoir-state covariance, as evidenced by the histogram of ridge-mode contributions as indicated in Fig.~\ref{fig:ridge_score_logistic}. In particular, the distributions of 
\[
\frac{\langle q_i,c\rangle^2}{\lambda_i+\lambda}
\]
are diffuse and lack concentration in a small number of dominant modes -- although the concentration mitigates the spectral pollution. This indicates that the covariance eigen-directions carrying significant variance are not well aligned with the prediction target, resulting in a spread of the ridge prediction score across many weakly informative modes. Consequently, neither construction of $B$ yields a low-dimensional representation of the predictive signal, leading to poor generalization in the autonomous forecasting task.

\begin{figure}[h]
\label{fig:ridge_score_logistic}
\centering
   \includegraphics[width=12cm]{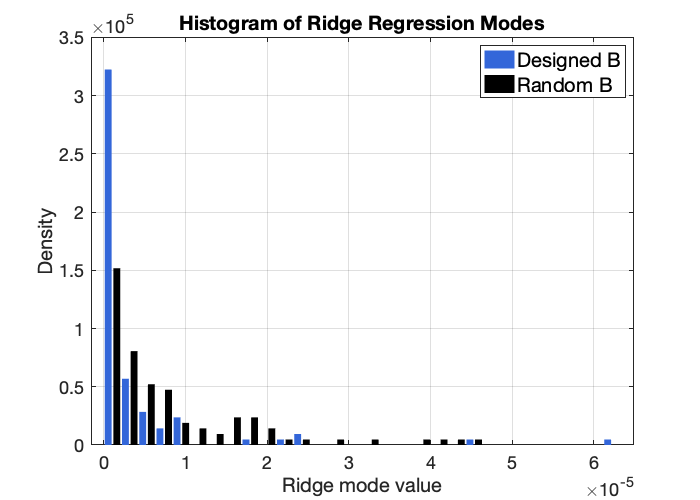}
 \caption{Full logistic map: Normalized histogram of ridge regression mode contributions 
$\langle q_i,c\rangle^2/(\lambda_i+\lambda)$, where $\{\lambda_i,q_i\}$ are the eigenpairs of the empirical state covariance matrix $\Sigma$ and $c$ is the state--target cross-covariance vector. The blue bars correspond to the data-aligned (designed) reservoir matrix $B$, while the black bars correspond to a random orthogonal $B$.}
\end{figure}

\end{document}